\documentclass[11pt, notitlepage]{article}
\usepackage{amssymb,amsmath,comment}
\catcode`\@=11 \@addtoreset{equation}{section}
\def\thesection{\arabic{section}}

\def\theequation{\thesection.\arabic{equation}}
\catcode`\@=12
\usepackage{colortbl}
\usepackage[mathscr]{eucal}
\usepackage{epsf}
\usepackage{a4wide}

\newcommand{\e}{\epsilon}

\newcommand{\Om} {\Omega}

\newcommand{\De} {\Delta}
\newcommand{\la} {\lambda}

\newcommand{\noi} {\noindent}

\newcommand{\mb} {\mathbb}
\newcommand{\mc} {\mathcal}

\usepackage[all]{xy}
\catcode`\@=11
\def\theequation{\@arabic{\c@section}.\@arabic{\c@equation}}
\catcode`\@=12

\newcommand{\QED}{\rule{2mm}{2mm}}

\newtheorem{Theorem}{Theorem}[section]
\newtheorem{Lemma}[Theorem]{Lemma}
\newtheorem{Proposition}[Theorem]{Proposition}
\newtheorem{Corollary}[Theorem]{Corollary}
\newtheorem{Remark}[Theorem]{Remark}
\newtheorem{Definition}[Theorem]{Definition}

\begin{document}

{\vspace{0.01in}}

\title
{ \sc On Concentration of least energy solutions for magnetic critical Choquard equations}

\author{%J. Giacomoni\footnote{LMAP (UMR CNRS 5142) Bat. IPRA,
 % Avenue de l'Universit\'e
   %F-64013 Pau, France. email:jacques.giacomoni@univ-pau.fr},
    ~~ T. Mukherjee\footnote{Department of Mathematics, Indian Institute of Technology Delhi,
Hauz Khaz, New Delhi-110016, India.
 e-mail: tulimukh@gmail.com}~ and ~K. Sreenadh\footnote{Department of Mathematics, Indian Institute of Technology Delhi,
Hauz Khaz, New Delhi-110016, India.
 e-mail: sreenadh@gmail.com} }

\date{}

\maketitle

\begin{abstract}

\noi In the present paper, we consider the following magnetic nonlinear Choquard equation
$$
\left\{
	\begin{array}{ll}
          & (-i \nabla+A(x))^2u + \mu g(x)u = \la u + (|x|^{-\alpha} * |u|^{2^*_\alpha})|u|^{2^*_\alpha-2}u \;\text{in} \; \mb R^n ,\\
 & u \in H^1(\mb R^n, \mb C)
	\end{array}
\right.$$
where $n \geq 4$, $2^*_\alpha= \frac{2n-\alpha}{n-2}$, { $\alpha \in (0,n)$}, ${\mu} >0$, ${{\la} >0 }$ is a parameter, $A(x): \mb R^n \rightarrow \mb R^n$ is a magnetic vector potential and $g(x)$ is a real valued potential function on $\mb R^n$. Using variational methods, we establish the existence of least energy solution under some suitable conditions. Moreover, the concentration behavior of solutions is also studied as $\mu \rightarrow +\infty$.

\medskip

\noi \textbf{Key words:}  Nonlinear Schr\"{o}dinger equations, Magnetic potential, Choquard equation, {Hardy-Littlewood-Sobolev critical} exponent.

\medskip

\noi \textit{2010 Mathematics Subject Classification:} 35R11, 35R09, 35A15.

\end{abstract}

\section{Introduction}
In this article, we study the existence and concentration behavior of nontrivial solutions of the following nonlinear Schr\"{o}dinger equation with nonlocal nonlinearity {of Choquard type}
$$(P_{\la,\mu})
\left\{
	\begin{array}{ll}
          (-i \nabla+A(x))^2u + \mu g(x)u = \la u + (|x|^{-\alpha} * |u|^{2^*_\alpha})|u|^{2^*_\alpha-2}u & \mbox{in } \mb R^n \\
          u \in H^1(\mb R^n, \mb C)
	\end{array}
\right.$$
where $n \geq 4$, $2^*_\alpha= \frac{2n-\alpha}{n-2}$, ${\alpha \in (0,n)}$, $\mu>0$, ${\la>0}$, $A = (A_1, A_2, \ldots, A_n): \mb R^n \rightarrow {\mb R^n}$ is a vector(or magnetic) potential such that {$A \in L^n_{\text{loc}}(\mb R^n, \mb R^n)$} and {$A$ is continuous at ${0}$} and $g(x)$ satisfies the following assumptions:
\begin{enumerate}
\item[(g1)] $g \in C(\mb R^n,\mb R)$, $g \geq 0$ and $\Om := \text{interior of}\;g^{-1}(0)$ is  a nonempty bounded set with smooth boundary and $\overline{\Om}= g^{-1}(0)$.
\item[(g2)] There exists $M>0$ such that ${\mc L}\{x \in \mb R^n:\; g(x)\leq M\} < {+\infty}$, where ${\mc L}$ denotes the Lebesgue measure in $\mb R^n$.
\end{enumerate}
{A} more general form of the above problem {is}
\begin{equation}\label{intro1}
(-i \nabla+A(x))^2w + G(x)w= F(x,w),\; w \in H^1(\mb R^n, \mb C)
\end{equation}
{which} arises when we try to look for standing wave solution of the Schr\"{o}dinger equation
\[i\hbar \frac{\partial \phi}{\partial t}= (-i \hbar \nabla+A(x))^2\phi + Q(x)\phi- n(x,|\phi|)\phi,\]
where $\hbar$ is the Plank constant. A lot of  attention has been paid to nonlinear Schr\"{o}dinger equation in recent years. When $A\equiv 0$( i.e. no magnetic potential) in \eqref{intro1}, many authors studied the problem {as in} \cite{AMS, BDP, CL}. The problem of the type
\begin{equation}\label{intro2}
-\Delta u + \mu a(x)u= \la u + |u|^{p-2}u,
\end{equation}
where $a \geq 0$ is potential well, with subcritical growth i.e. $p < 2^*= 2n/(n-2)$ has been investigated extensively in \cite{ABC,BW,BPW,Oh,Wang}. In {the} critical case $p=2^*$, Clapp and Ding {in} \cite{CD} established {the} existence and multiplicity of positive solutions {for \eqref{intro2}} using variational methods. For Schr\"{o}dinger equations with critical nonlinearity, one may also refer \cite{AF,CS,DL,TS_sing}. {In \cite{GM}, authors have studied the blow-up of radial solutions to a cubic nonlinear Schr\"{o}dinger equation with a radial defect, located on the sphere of radius $r_0$}. {We also suggest readers to refer \cite{AFY,BR,HKT} for further study.}

When the magnetic vector potential $A \not \equiv 0$, the Schr\"{o}dinger equation of the form
\begin{equation*}
(-i\hbar \nabla +A(x))^2u + V(x)u = |u|^{p-2}u \; \text{in}\; \mb R^n,
\end{equation*}
where $V$ is electric potential function, has been widely studied by many authors, {we} refer {\cite{CT,CS1,LPW,Sq1,Sq2}} {for this} and the references therein. Motivated by these results, very recently L\"{u} \cite{Lu} studied the problem
\begin{equation}\label{intro4}
(-i \nabla +A(x))^2u + (g_0(x)+\mu g(x))(x)u = (|x|^{-\alpha}* |u|^p)|u|^{p-2}u, \; u \in H^1(\mb R^n, \mb C),
\end{equation}
where $n\geq 3$, $\alpha \in (0,n)$, $p \in \left( \frac{2n-\alpha}{n}, \frac{2n-\alpha}{n-2}\right)$, $g_0$ and $g$ are real valued functions on $\mb R^n$ satisfying some {necessary} conditions and $\mu >0$.  He proved the existence of ground state solution {when} $\mu \geq \mu^*$, for some $\mu^*>0$ and concentration behavior of solutions as $\mu \rightarrow \infty$. The Hardy-Littlewood-Sobolev inequality (see  Theorem \ref{H-L_ineq}) plays an important role for studying such problems and in that context, we {call} $2^*_\alpha = \frac{2n-\alpha}{n-2}$ {as} the crtical exponent {in the sense of Hardy-Littlewood-Sobolev inequality}. When $A \equiv 0$, $g_0\equiv 0,\;g\equiv 1$ and $\mu =1$ in \eqref{intro4}, that is
\[-\Delta u + u = (|x|^{-\alpha}* |u|^p)|u|^{p-2}u, u \in H^1(\mb R^n)\]
the equations are generally called the Choquard equations which arise in various fields of physics, example quantum theory of large systems of nonrelativistic bosonic atoms and molecules. Choquard equations are another topic of attraction for researchers now a days which in turn rendered a huge literature in this area, {for instance refer} \cite{clsa,ghsc,TS}.  In \cite{LB}, Lieb  proved the existence and uniqueness of solution, up to translations, for the problem
\[ -\De u+ u = (|x| * |u|^2)|u|^2 \; \text{in}\; \mb R^n. \]
%where  $f(t)$ is critical growth nonlinearity such that $|tf(t)| \leq C||t|^2 + |t|^{\frac{2n-\mu}{n-2}}|$ for $t \in \mb R$, $\mu >0$, for some $C>0$ and $F(t) = \int_0^{z}f(z) \mathrm{d}z$.
 In \cite{ my3,my, my1}, Gao and Yang  showed existence and multiplicity results for Brezis-Nirenberg type  problem {for} the nonlinear  Choquard equation
\begin{equation*}
-\De u = \left( \int_{\Om}\frac{|u|^{2^*_{\alpha}}}{|x-y|^{\alpha}}\mathrm{d}y \right)|u|^{2^*_{\alpha}-2}u + \la g(u) \; \text{in}\;
\Om, \quad u = 0 \; \mbox{on}\; \partial\Om,
\end{equation*}
where $\Om$ is smooth bounded domain in $\mb R^n$, $n>2$, $\la>0$, $0<\alpha<n$ and $g(u)$ is a nonlinearity with certain {necessary} assumptions. Salazar {in} \cite{szr} showed existence of vortex type solutions for the stationary nonlinear magnetic Choquard equation
\begin{equation*}
(-i \nabla +A(x))^2u + W(x)u = (|x|^{-\alpha}* |u|^p)|u|^{p-2}u \; \text{in} \; \mb R^n,
\end{equation*}
where $n\geq 3$, $\alpha \in (0,n)$, $p \in \left[2, 2^*_\alpha \right)$, $A: \mb R^n \to \mb R^n$ is magnetic potential and $W: \mb R^n \to  \mb R$ is bounded electric potential (under some assumptions on decay of $A$ and $W$ at infinity). Cingloni, Sechi and Squassina showed existence of family of solutions for a Schr\"{o}dinger equation in the presence of electric and magnetic potential and Hartree-type nonlinearity in \cite{CSS}. Schr\"{o}dinger equations with magnetic field and  Choquard type nonlinearity has also been studied in \cite{MSS,MSS1}. In this context, we also cite \cite{agm,ags,YW} with no {attempt} to provide the full list {of references}.

 Now a very obvious question arises, what happens in the critical case i.e. when $p=2^*_\alpha$ in \eqref{intro4}? Here in this paper, we consider the problem $(P_{\la,\mu})$ which is motivated by \eqref{intro2} and \eqref{intro4}. The main difficulty for this problem is the presence of {critical nonlinearity} in the sense of Hardy-Littlewood-Sobolev inequality which is also nonlocal in nature. The critical exponent term being nonlocal adds on the difficulty to study the Palais-Smale level around a nontrivial critical point. We define $\nabla_A u := (-i\nabla +A(x))u$ and consider the minimization problem here by defining
\[S_A = {\inf_{u \in H^1_A(\mb R^n) \setminus \{0\}} \frac{\int_{\mb R^n}|\nabla_Au|^2~\mathrm{d}x}{\int_{\mb R^n}(|x|^{\alpha}* |u|^{2^*_\alpha})|u|^{2^*_\alpha}~\mathrm{d}x}}\]
and proved that it is attained under some necessary and sufficient conditions which is a new result. Also the other results proved here are {completely} new and there is no work concerning this problem {till now} to the best of our knowledge. Following the approach of \cite{CD}, we show that $(P_{\la,\mu})$ has a solution. Also we show that the problem
$$ (P_\la)
\left\{
	\begin{array}{ll}
        (-i\nabla +A(x))^2u = \la u + (|x|^\alpha*|u|^{2^*_\alpha})|u|^{2^*_\alpha-2}u & \mbox{in } \Om, \\
        u=0 & \mbox{on } \partial \Om
	\end{array}
\right.$$
for small $\la$ acts as a limit problem for $(P_{\la,\mu})$ as $\mu \rightarrow \infty$. We use the knowledge of $(P_\la)$ to show the concentration behavior of solutions of $(P_{\la,\mu})$.

\noi {We divide our paper into 4 sections. Section 2 contains the variational setting and the main results of our work. We study the Palais-Smale sequences and proved some compactness results in section 3. Making {use} of these results, we establish the proof of main theorems in section 4.}

\section{Variational setting and main results}
We assume that {$g$ satisfies the conditions} $(g1)$ and $(g2)$ throughout this paper. Let us define
$$ H^1_A(\mb R^n, \mb C)= \left\{u \in L^2(\mb R^n,\mb C) \;: \; \nabla_A u \in L^2(\mb R^n, \mb C^n)\right\}.$$
Then $H^1_A(\mb R^n, \mb C)$ is a Hilbert space with the inner product
$$\langle u,v\rangle_A = \text{Re} \left(\int_{\mb R^n} (\nabla_A u \overline{\nabla_A v} + u \overline v)~\mathrm{d}x \right),$$
where $\text{Re}(w)$ denotes the real part of $w \in \mb C$ and $\bar w$ denotes its complex conjugate. The associated norm $\|\cdot\|_A$ on the space $H^1_A(\mb R^n, \mb C)$ is given by
$$\|u\|_A= \left(\int_{\mb R^n}(|\nabla_A u|^2+|u|^2)~\mathrm{d}x\right)^{\frac{1}{2}}.$$
We call $H^1_A(\mb R^n, \mb C)$ simply $H^1_A(\mb R^n)$. Let $H^{0,1}_A(\Om, \mb C)$ (denoted by $H^{0,1}_A(\Om)$ for simplicity) be the Hilbert space defined by the closure of $C_c^{\infty}(\Om, \mb C)$ under the scalar product {$\langle u,v \rangle_A= \textstyle\text{Re}\left(\int_{\Om}(\nabla_A u \overline{\nabla_A v}+u \overline v)~\mathrm{d}x\right)$}, where $\Om = \text{interior of } g^{-1}(0)$. Thus {norm on $H^{0,1}_A(\Om)$ is given by}
 \[\|u\|_{H^{0,1}_A(\Om)}= \left(\int_\Om (|\nabla_A u|^2+|u|^2)~\mathrm{d}x\right)^{\frac{1}{2}}.\]
 Let $E = \left\{u \in H^1_A(\mb R^n): \int_{\mb R^n}g(x)|u|^2~\mathrm{d}x < +\infty\right\}$ be the Hilbert space equipped with the inner product
$$\langle u,v\rangle = \text{Re} \left(\int_{\mb R^n}\left(\nabla_A u \overline{\nabla_A v}~\mathrm{d}x + g(x)u\bar v \right)~\mathrm{d}x\right)$$
and the associated norm $\|\cdot\|_E$, where
$$\|u\|_E^2= \int_{\mb R^n}\left(|\nabla_A u|^2+ g(x)|u|^2\right)~\mathrm{d}x. $$
Then $\|\cdot\|_E$ is clearly equivalent to each of the norm $\|\cdot\|_\mu$, where
$$\|u\|_\mu^2= \int_{\mb R^n}\left(|\nabla_A u|^2+ \mu g(x)|u|^2\right)~\mathrm{d}x$$
for $\mu >0$. We have the following well known \textit{diamagnetic inequality} (for detailed proof, see \cite{LL}, Theorem $7.21$ ).

\begin{Theorem}\label{dia_eq}
If $u \in H^1_A(\mb R^n)$, then $|u| \in H^1(\mb R^n,\mb R)$ and
$$|\nabla |u|(x)| \leq |\nabla u(x)+ i A(x)u(x)| \; \text{for a.e.}\; x \in \mb R^n.$$
\end{Theorem}
{\begin{proof}
{The outline of the proof is as follows: since $A:\mb R^n \to \mb R^n$ we get}
\[\left| \nabla |u|(x)\right| = \left|\text{Re}\left(\nabla u \frac{\bar u}{|u|}\right)\right|= \left|\text{Re}\left( (\nabla u+ i Au)\frac{\bar u}{|u|}\right)\right|\leq |\nabla u+ i Au| .\]
\hfill{\QED}
\end{proof}}

\noi  So for each $q \in [2,2^*]$, there exists constant $b_q>0$ (independent of $\mu$) such that
\begin{equation}\label{mg-eq1}
|u|_q \leq b_q\|u\|_\mu, \; \text{for any}\; u \in E,
\end{equation}
where $|\cdot|_q$ denotes the norm in $L^q(\mb R^n,\mb C)$ and $2^*= \textstyle\frac{2n}{n-2}$ is the Sobolev critical exponent. {Also $H^1_A(\Om) \hookrightarrow L^q(\Om, \mb C)$ is continuous for each $1\leq q \leq 2^*$ and compact when $1\leq q < 2^*$.} Let us denote
\begin{equation}\label{B(u)}
B(u)= \int_{\mb R^n}(|x|^{\alpha}*|u|^{2^*_\alpha})|u|^{2^*_\alpha}~\mathrm{d}x = \int_{\mb R^n}\int_{\mb R^n}\frac{|u(x)|^{2^*_\alpha}|u(y)|^{2^*_\alpha}}{|x-y|^\alpha}~\mathrm{d}x\mathrm{d}y.
\end{equation}
To estimate the nonlocal term $B(u)$, we have the following Hardy-Littlewood-Sobolev inequality (refer \cite{LL}, Theorem 4.3).
 \begin{Proposition}\label{H-L_ineq}
 Let $t,r>1$ and $0<\alpha<n $ with $1/t+ {\alpha/n} +1/r=2$, $f \in L^t(\mathbb R^n)$ and $h \in L^r(\mathbb R^n)$. There exists a sharp constant $C(t,n,\alpha,r)$, independent of $f,h$ such that
 \begin{equation}\label{HLSineq}
 \int_{\mb R^n}\int_{\mb R^n} \frac{f(x)h(y)}{|x-y|^{\alpha}}\mathrm{d}x\mathrm{d}y \leq C(t,n,\alpha,r)|f|_t|h|_r.
 \end{equation}
 { If $t =r = \textstyle\frac{2n}{2n-\alpha}$ then
 \[C(t,n,\alpha,r)= C(n,\alpha)= \pi^{\frac{\alpha}{2}} \frac{\Gamma\left(\frac{n}{2}-\frac{\alpha}{2}\right)}{\Gamma\left(n-\frac{\alpha}{2}\right)} \left\{ \frac{\Gamma\left(\frac{n}{2}\right)}{\Gamma(n)} \right\}^{-1+\frac{\alpha}{n}}.  \]
 In this case there is equality in \eqref{HLSineq} if and only if $f\equiv (constant)h$ and
 \[h(x)= {z}(\gamma^2+ |x-a|^2)^{\frac{-(2n-\alpha)}{2}}\]
 for some ${z} \in \mathbb C$, $0 \neq \gamma \in \mathbb R$ and $a \in \mathbb R^n$.}
 \end{Proposition}
Proposition \ref{H-L_ineq} implies that
\begin{equation}\label{h-l1}
{|B(u)| \leq C(n,\alpha)|u|_{2^*}^{22^*_\alpha},}
\end{equation}
where ${C(n,\alpha)}$ is as given in Proposition \ref{H-L_ineq}. By \eqref{mg-eq1}, we say that $B(u)$ is well defined for $u \in E$. Also $B(u)\in C^1(E,\mb R)$, refer Lemma $2.5$ of \cite{YW}.
\begin{Definition}
We say that a function $u \in {E}$ is a weak solution of $(P_{\la,\mu})$ if
\[\text{Re}\left( \int_{\mb R^n} \nabla_A u \overline{\nabla_A v}~\mathrm{d}x + \int_{\mb R^n}(\mu g(x)-\la )u\overline{v}~\mathrm{d}x- \int_{\mb R^n} (|x|^{-\alpha}* |u|^{2^*_\alpha})|u|^{2^*_\alpha-2}u \overline{v}~\mathrm{d}x \right)=0\]
for all $v \in {E}$.
\end{Definition}
\begin{Definition}
A solution $u$ {of $(P_{\la,\mu})$} is said to be a least energy solution if the energy functional
\[I_{\la,\mu}(u)=  \int_{\mb R^n}\left( \frac12\left(|\nabla_A u|^2+ (\mu g(x)-\la)|u|^2\right) - \frac{1}{22^*_\alpha}(|x|^{-\alpha}* |u|^{2^*_\alpha})|u|^{2^*_\alpha}\right) ~\mathrm{d}x\]
achieves its minimum at $u$ over all the nontrivial solutions of $(P_{\la,\mu})$.
\end{Definition}
\begin{Definition}
A sequence of solutions $\{u_k\}$ of $(P_{\la,\mu_k})$ is said to concentrate at a solution $u$ of $(P_\la)$ if a subsequence converges strongly to $u$ in $H^1_A(\mb R^n)$ as $\mu_k \rightarrow \infty$.
\end{Definition}
The main idea to prove the existence of solution for the problem $(P_{\la,\mu})$ is using variational methods where the weak solutions for $(P_{\la,\mu})$ are obtained by finding the critical points of the energy functional $I_{\la,\mu}: H^1_A(\mb R^n) \rightarrow \mb R$ defined by
\[I_{\la,\mu}(u)=  \int_{\mb R^n}\left( \frac12\left(|\nabla_A u|^2+ (\mu g(x)-\la)|u|^2\right) - \frac{1}{22^*_\alpha}(|x|^{-\alpha}* |u|^{2^*_\alpha})|u|^{2^*_\alpha}\right) ~\mathrm{d}x.\]
Then $I_{\la,\mu} \in C^1(E, \mb R)$ with
\[\langle I^{\prime}_{\la,\mu}(u),v\rangle = \text{Re}\left( \int_{\mb R^n} \nabla_A u \overline{\nabla_A v}~\mathrm{d}x + \int_{\mb R^n}\left(\mu g(x)-\la -  (|x|^{-\alpha}* |u|^{2^*_\alpha})|u|^{2^*_\alpha-2}\right)u \overline{v}~\mathrm{d}x \right)\]
for $u,v \in E$. Thus we characterize the weak solutions of $(P_{\la,\mu})$ as the critical points of $I_{\la,\mu}$. From now onwards, we denote ${\la_1(\Om)}>0$ as the best constant of the compact embedding ${H^{0,1}_A(\Om)}\hookrightarrow L^2(\Om, \mb C) $ given by
\[ \la_1(\Om) = \inf\limits_{u \in H^{0,1}_A(\Om)}\left\{\int_{\Om}|\nabla_A u|^2 ~\mathrm{d}x : \; {\int_{\Om}|u|^2~\mathrm{d}x}=1\right\}\]
which is also the first eigenvalue of $-\Delta_A := (-i\nabla +A)^2$ on $\Om$ with boundary condition $u=0$. Let $S$ denote the best Sobolev constant of the embedding $H^1_0(\Om,\mb R) \hookrightarrow L^{2^*}(\Om,\mb R) $ which is given by
\[S = \inf\limits_{u \in H^1_0(\Om,\mb R)}\left\{\int_{\Om}|\nabla u|^2 ~\mathrm{d}x:\int_{\Om}|u|^{2^*}~\mathrm{d}x=1 \right\}. \]
We know that $S$ is independent of $\Om$ and it is achieved if and only if $\Om = \mb R^n$. We use $S_{H,L}$ to denote the best constant {as}
\begin{equation*}
S_{H,L}= \inf\limits_{u \in H^1(\mb R^n,\mb R)} \left\{\int_{\Om}|\nabla u|^2 ~\mathrm{d}x: B(u)=1\right\}.
\end{equation*}
By Lemma $1.2$ of \cite{my}, we get  that $S_{H,L}$ is achieved by {functions of the form}
\[U(x)= C\left(\frac{b}{b^2+ |x-a|^2}\right)^{\frac{n-2}{2}}\]
where {$C>0$ is a fixed constant, $a \in \mb R^n$ and $b>0$ are parameters}. {Now we state our main results :}

\begin{Theorem}\label{MT1}
For every $\la \in (0, \la_1(\Om))$ there exists a $\mu(\la)>0$ such that $(P_{\la,\mu})$ has a least energy solution $u_\mu$ for each $\mu\geq \mu(\la)$.
\end{Theorem}

\begin{Theorem}\label{MT2}
Let $\{u_m\}$ be a sequence of non-trivial solutions of $(P_{\la,\mu_m})$ with $\mu_m \rightarrow \infty$ and $I_{\la,\mu_m}(u_m) \rightarrow c< \frac{n+2-\alpha}{2(2n-\alpha)}S_A^{\frac{2n-\alpha}{n+2-\alpha}}$ as $m \rightarrow \infty$. Then $u_{m}$ concentrates at a solution of $(P_\la)$.
\end{Theorem}

\section{Palais Smale analysis and compactness results}
In this section, we find the Palais Smale critical threshold below which any Palais Smale $(PS)_c$ sequence has a convergent subsequence. We recall that a sequence $\{u_m\} \subset E$ is said to be a $(PS)_c$ sequence (for $I_{\la,\mu}$) if $I_{\la,\mu}(u_m) \rightarrow c$ and $I^{\prime}_{\la,\mu}(u_m) \rightarrow 0$ as $m \rightarrow \infty$. We say that $I_{\la,\mu}$ satisfies the $(PS)_c$ condition if every $(PS)_c$ sequence contains a convergent subsequence.
\begin{Lemma}\label{comp_lem1}
Suppose $\mu_m \geq 1$ and $u_m \in E$ be such that $\mu_m \rightarrow \infty$ as $m \rightarrow \infty$ and {there exists a $K>0$ such that} $\|u_m\|_{\mu_m} < K$, for all $m \in \mb N$. Then there exists a $u \in H^{0,1}_A(\Om)$ such that (upto a subsequence), $u_m \rightharpoonup u$ weakly in $E$ and $u_m \rightarrow u$ strongly in $L^2(\mb R^n)$  as $m \rightarrow \infty$.
\end{Lemma}
\begin{proof}
Since the norms $\|\cdot\|_E$ and $\|\cdot\|_\mu$ are equivalent, we have $\|u_m\|_E^2 < K^\prime$, for some constant $K^\prime >0$. So there exists $u \in E$ such that $ u_m \rightharpoonup u$ weakly in $E$ and $u_m \rightarrow u$ strongly in $L^2_{\text{loc}}(\mb R^n)$ as $m \rightarrow \infty$. Let $C_r = \{x : |x|\leq r,\; g(x) \geq 1/r\}$, $r \in \mb N$. Then we can easily see that
\[\int_{C_r}|u_m|^2~\mathrm{d}x \leq r\int_{C_r} g(x)|u_m|^2~\mathrm{d}x \leq \frac{r}{\mu_m}\|u_m\|^2_{\mu_m} \leq \frac{rK}{\mu_m} \rightarrow 0 \; \text{as} \; m \rightarrow \infty. \]
This holds for every $r$ which implies that $u \equiv 0$ in $\mb R^n\setminus \Om$. Since $\partial \Om$ is smooth, we have $u \in H^{0,1}_A(\Om)$. The next step is to show that $u_m \rightarrow u$ strongly in $L^2({\mb R^n})$. Let $D= \{x \in \mb R^n: \; g(x) \leq M\}$, where $M$ is defined as in $(g2)$. Then
\begin{equation}\label{comp1}
\int_{\mb R^n \setminus D}u_m^2~\mathrm{d}x \leq  \frac{1}{\mu_m M} \int_{\mb R^n \setminus D}\mu_m g(x)u_m^2 ~\mathrm{d}x \leq \frac{K^\prime}{\mu_m M} \rightarrow 0 \; \text{as}\; m \rightarrow \infty.
\end{equation}
Let $B_R = \{x \in \mb R^n :\; |x| \leq R\}$ and $q \in \left(1, \frac{n}{n-2}\right)$ such that $q^\prime = \frac{q}{q-1}$. Then using \eqref{mg-eq1} and equivalence of norms {$\|\cdot\|_\mu$ and $\|\cdot\|_E$}, we have
\[\int_{B_R^c \cap D} |u_m-u|^2~\mathrm{d}x \leq |u_m-u|^2_{2q} (\mc L(B_R^c \cap D))^{1/q^\prime} \leq C_1 b_{2q}^2 \|u_m-u\|_E^2(\mc L(B_R^c \cap D))^{1/q^\prime},\]
where $C_1$ is a positive constant and $B_R^c= \mb R^n \setminus B_R$. Hence by $(g2)$ we get
\begin{equation}\label{comp2}
\int_{B_R^c \cap D} |u_m-u|^2~\mathrm{d}x \rightarrow 0\; \text{as}\; R \rightarrow \infty.
\end{equation}
 Lastly, as we know $u_m \rightarrow u $ strongly in $L^2_{\text{loc}}(\mb R^n)$ we get
\begin{equation}\label{comp3}
\int_{B_R}|u_m -u|^2 ~\mathrm{d}x \rightarrow 0 \; \text{as}\; m \rightarrow \infty.
\end{equation}
Therefore using \eqref{comp1}, \eqref{comp2} and \eqref{comp3}, we get $u_m \rightarrow u$ strongly in $L^2(\mb R^n)$ as $m \rightarrow \infty$.\hfill{\QED}
\end{proof}

\noi Now let $T_\mu := -\Delta_A + \mu g(x)$ , where $-\De_A = (-i\nabla +A)^2$, be an operator defined on $E$. {Also let $v \in E$ and denote $(\cdot,\cdot)$ as $L^2$ inner product then we write}
\[\big(T_\mu(u),v\big)= \text{Re}\left(\int_{\mb R^n}(\nabla_A u \overline{\nabla_A v}+ \mu g(x)u\overline v)~\mathrm{d}x\right).\]
Clearly $T_\mu$ is a self adjoint operator and if $a_\mu := \inf \sigma(T_\mu)$, i.e. the infimum of the spectrum of $T_\mu$, then $a_\mu$ can be characterized as
\[0 \leq a_\mu = \inf \{\big(T_\mu(u),u\big): \; u \in E,\; |u|_2=1\}= \inf \{\|u\|_\mu^2:\; u\in E,\;|u|_2=1\}.\]
Thus $a_\mu $ is nondecreasing in $\mu$. Therefore we get
\[\big((T_\mu-\la)u,u\big) = \int_{\mb R^n}(|\nabla_A u|^2 + \mu g(x)|u|^2- \la |u|^2)~ \mathrm{d}x.\]

%%%%%%%%%%%%%%%%%%%%%%%%%%%%%%%%%%%%%%%%%%%%%%%%%%%%%%%%%%%%%%%%%%%%%%%%%%%%%%%%%%%%%%%%%%%%%%%%%%%%%%%%%%%%%%%%%%%%%%%%%%%%%%%%%%%%%%%%%%%%%%%%55
\noi In the next lemma, we will show that the map $(T_\mu -\la)$ is coercive.
\begin{Lemma}\label{comp_lem2}
For each $\la \in (0, \la_1(\Om))$, there exists a $\mu(\la)>0$  such that $a_\mu \geq (\la +\la_1(\Om))/2$ whenever $\mu \geq \mu(\la)$. As a consequence
\[\big((T_\mu-\la)u,u) \geq \beta_\la \|u\|_\mu^2\]
for all $u \in E$, $\mu \geq \mu(\la)$, where $\beta_\la := (\la_1(\Om)-\la)/(\la_1(\Om)+\la)$.
\end{Lemma}
\begin{proof}
Assume by contradiction that there exist a sequence $\mu_m \rightarrow \infty$ such that $a_{\mu_m}< (\la +\la_1(\Om))/2$ for all $m$ and $a_{\mu_m} \rightarrow \theta \leq (\la +\la_1(\Om))/2$. Let us consider a minimizing sequence $\{u_m\}\in E$ such that $|u_m|_2 =1$ and $((T_{\mu_m}-a_{\mu_m})u_m,u_m)\rightarrow 0$. This implies
\begin{equation*}
\begin{split}
\|u_m\|_{\mu_m}^2 &= \int_{\mb R^n} (|\nabla_Au|^2 + \mu_m g(x)|u_m|^2)~\mathrm{d}x\\
&= \big((T_{\mu_m}-a_{\mu_m})u_m,u_m\big)+ a_{\mu_m}(u_m,u_m)\\
& \leq \big((T_{\mu_m}-a_{\mu_m})u_m,u_m\big)+ (1+a_{\mu_m})|u_m|^2_2\\
& \leq {2}(1+\la_1(\Om))
\end{split}
\end{equation*}
for large $m$, using $\la < \la_1(\Om)$ and $\theta \leq (\la +\la_1(\Om))/2 $. So using Lemma \ref{comp_lem1}, we get $u \in H^{0,1}_A(\Om)$ such that  $u_m \rightharpoonup u$ weakly in $E$ and $u_m \rightarrow u$ strongly in $L^2(\mb R^n)$ as $ m \rightarrow \infty$. Therefore $|u|_2=1$ and $\liminf\limits_{m\rightarrow \infty}|\nabla_A u_m|^2_2 \geq |\nabla_A u|_2^2$. Since $g\geq 0$ and $\mu_m \rightarrow \infty$, we have
\begin{equation*}
\begin{split}
\int_{\Om}(|\nabla_Au|^2- \theta |u|^2) &\leq \liminf\limits_{m \rightarrow \infty}\int_{\mb R^n} (|\nabla_Au_m|^2 + \mu_m g(x)|u_m|^2- a_{\mu_m}|u_m|^2 )~\mathrm{d}x\\
&= \liminf\limits_{m \rightarrow \infty}\big((T_{\mu_m}-a_{\mu_m})u_m, u_m\big) =0.
\end{split}
\end{equation*}
Hence
\[\int_{\Om}|\nabla_Au|^2~\mathrm{d}x \leq \theta \leq \frac{\la+\la_1(\Om)}{2}< \la_1(\Om)\]
which is a contradiction to the definition of $\la_1(\Om)$. {Therefore there exists a  $\mu(\la)>0$  such that $a_\mu \geq (\la +\la_1(\Om))/2$ whenever $\mu \geq \mu(\la)$.} For the second part, let $u \in E$ and $\mu \geq \mu(\la)$ then  $a_\mu \leq \frac{\|u\|_\mu^2}{|u|^2_2}$
which gives $$\la|u|^2_2 \leq \frac{2\la\|u\|_\mu^2}{\la+\la_1(\Om)}.$$ Therefore
\[\big((T_\mu-\la)u,u\big)\geq \|u\|_\mu^2-\la|u|^2_2 \geq \frac{\la_1(\Om)-\la}{\la_1(\Om) +\la} \|u\|_\mu^2.\]
\hfill{\QED}
\end{proof}

\noi Our next lemma assures that all $(PS)_c$ sequences are bounded.
%%%%%%%%%%%%%%%%%%%%%%%%%%%%%%%%%%%%%%%%%%%%%%%%%%%%%%%%%%%%%%%%%%%%%%%%%%%%%%%%%%%%%%%%%%%%%%%%%%%%%%%%%%%%%%%%%%%%%%%%%%%%%%%%%%%%%%%%%%%
\begin{Lemma}\label{comp_lem3}
Let $\{u_m\}$ be a $(PS)_c$ sequence for $I_{\la,\mu}$. If $\la \in (0,\la_1(\Om))$ and $\mu \geq \mu(\la)$, then $\{u_m\}$ is bounded in $E$ and
\[\lim\limits_{m \rightarrow \infty}\big((T_\mu-\la)u_m,u_m\big)= \lim\limits_{m \rightarrow \infty} B(u_m) = \frac{2c(2n-\alpha)}{(n+2-\alpha)},\]
where $B(\cdot)$ is defined in \eqref{B(u)}.
\end{Lemma}
\begin{proof}
Using definitions of $I_{\la,\mu}$ and $T_\mu$, we get
\begin{equation}\label{PSbdd1}
\begin{split}
I_{\la,\mu}(u_m)- \frac{1}{22^*_\alpha} (I_{\la,\mu}^{\prime}(u_m),u_m) &= \left( \frac{2^*_\alpha-1}{22^*_\alpha}\right)  \int_{\mb R^n}(|\nabla_Au|^2 + \mu g(x)|u_m|^2 - \la |u_m|^2)~\mathrm{d}x\\
& = \frac{n+2-\alpha}{2(2n-\alpha)}\big((T_\mu-\la)u_m,u_m\big)
\end{split}
\end{equation}
and
\begin{equation}\label{PSbdd2}
I_{\la,\mu}(u_m)- \frac{1}{2} \langle I_{\la,\mu}^{\prime}(u_m),u_m \rangle = \left(\frac{2^*_\alpha-1}{22^*_\alpha}\right) B(u_m)= \frac{n+2-\alpha}{2(2n-\alpha)} B(u_m).
\end{equation}
Using Lemma \ref{comp_lem2} and \eqref{PSbdd1}, we get
\begin{equation*}
\begin{split}
c-\frac{1}{22^*_\alpha}o(\|u_m\|_\mu)  \geq \frac{n+2-\alpha}{2(2n-\alpha)} \big((T_\mu-\la)u_m,u_m\big) \geq {\frac{n+2-\alpha}{2(2n-\alpha)}} \beta_\la \|u_m\|_\mu^2.
\end{split}
\end{equation*}
This implies $\{u_m\}$ is {a bounded sequence} in $E$. Taking limit $m \rightarrow \infty$ in \eqref{PSbdd1}, we get
\[\lim\limits_{m \rightarrow \infty} {\big((T_\mu-\la)u_m,u_m\big)}  = \frac{2c(2n-\alpha)}{n+2-\alpha},\]
and taking limit $m \rightarrow \infty$ in \eqref{PSbdd2}, we get
\[\lim\limits_{m \rightarrow \infty} B(u_m) = \frac{2c(2n-\alpha)}{n+2-\alpha}.\]
This completes the proof.\hfill{\QED}
\end{proof}

%%%%%%%%%%%%%%%%%%%%%%%%%%%%%%%%%%%%%%%%%%%%%%%%%%%%%%%%%%%%%%%%%%%%%%%%%%%%%%%%%%%%%%%%%%%%%%%%%%%%%%%%%%%%%%%%%%%%%%%%%%%%%%%%%%%%%%%%
\noi Let
\begin{equation*}
S_A := \inf\limits_{{u \in H^1_A(\mb R^n)\setminus\{0\}}} \frac{\int_{\mb R^n}|\nabla_Au|^2 ~\mathrm{d}x}{B(u)^{\frac{n-2}{2n-\alpha}}}.
\end{equation*}
 For the preceding sections, {enlarging $\mu(\la)$} if necessary, we assume $\mu(\la)\geq \la/M$, where $M$ is defined in $(g2)$. Thus,
\begin{equation}\label{lambdd}
\mu M-\la \geq 0,\; \text{for all}\; \mu \geq \mu(\la).
\end{equation}
%Before proving our Proposition \ref{comp_lem4}, we recall Lemma $3.5$ of \cite{Ack} here.
%\begin{Lemma}\label{Lemma3.5}
%$\Psi$ is a $C^1$ functional where $\Psi$ and $\Psi^\prime$ map bounded sets to bounded sets. $\Psi$ is weakly sequentially lower semicontinuous and $\Psi^\prime$ is weakly sequentially semicontinuous. If $u_m \rightharpoonup v$ weakly in $E$, then there exists a subsequence $v_m\rightharpoonup v$ weakly such that
%$$ \Psi(u_m)- \Psi(u_m-v_m) \rightarrow \Psi(v)\; \text{in}\; \mb R$$
%$$  \Psi^\prime(u_m)- \Psi^\prime(u_m-v_m) \rightarrow \Psi^\prime(v)\; \text{in}\; E^*.$$
%\end{Lemma}
\begin{Proposition}\label{comp_lem4}
If $\la \in (0,\la_1(\Om))$ and $\mu \geq \mu(\la)$, then the functional $I_{\la,\mu}$ satisfies the $(PS)_c$ condition, for all $$c \in \left(-\infty,  \frac{n+2-\alpha}{2(2n-\alpha)} S_A^{\frac{2n-\alpha}{n+2-\alpha}}\right).$$
\end{Proposition}
\begin{proof}
Let $\{u_m\} \subset E$ be a sequence such that $I_{\la,\mu}(u_m)\rightarrow c< \frac{n+2-\alpha}{2(2n-\alpha)} S_A^{\frac{2n-\alpha}{n+2-\alpha}}$ and $I_{\la,\mu}^\prime(u_m) \rightarrow 0$ as $m \rightarrow \infty$. By Lemma \ref{comp_lem3}, $\{u_m\}$ is bounded in $E$ that is $\|u_m\|_\mu\leq K_1$ for some constant $K_1>0$ and for all $m$. Therefore, there {exists} a subsequence(still denoted by $\{u_m\}$) such that $u_m \rightharpoonup u$ weakly in $E$, $u_m \rightharpoonup u$ weakly in $L^{2^*}(\mb R^n)$,  $u_m \rightarrow u$ strongly in $L^2_{\text{loc}}(\mb R^n)$ and $u_m(x) \rightarrow u(x)$ a.e. for $x \in \mb R^n$.
% Using {\eqref{mg-eq1} and \eqref{h-l1} we get that} there exists a $K_2>0$ such that
%\[|B(u_m)| \leq K_2.\]
{This implies that as $m \to \infty$
\[|u_m|^{2^*_\alpha} \rightharpoonup |u|^{2^*_\alpha} \; \text{in}\; L^{\frac{2n}{2n-\alpha}}(\mb R^n)\; \text{and}\; |u_m|^{2^*_\alpha-2}u_m \rightharpoonup |u|^{2^*_\alpha-2}u \; \text{in}\; L^{\frac{2n}{n+2-\alpha}}(\mb R^n).\]
Therefore by Hardy-Littlewood-Sobolev inequality we get
\[|x|^{\alpha}\ast |u_m|^{2^*_\alpha} \rightharpoonup |x|^{\alpha}\ast |u|^{2^*_\alpha} \; \text{in}\; L^{\frac{2n}{\alpha}}(\mb R^n) \; \text{as}\; m \to \infty. \]
Hence we obtain that
\[(|x|^{\alpha}\ast |u_m|^{2^*_\alpha})|u_m|^{2^*_\alpha-2}u_m \rightharpoonup (|x|^{\alpha}\ast |u|^{2^*_\alpha})|u|^{2^*_\alpha-2}u \; \text{in}\; L^{\frac{2n}{n+2}}(\mb R^n) \; \text{as}\; m \to \infty. \]
This implies that for any $\phi \in E \subset L^{2^*}(\mb R^n)$
\[\int_{\mb R^n}(|x|^{\alpha}\ast |u_m|^{2^*_\alpha})|u_m|^{2^*_\alpha-2}u_m \phi~\mathrm{d}x \rightarrow \int_{\mb R^n}(|x|^{\alpha}\ast |u|^{2^*_\alpha})|u|^{2^*_\alpha-2}u \phi~\mathrm{d}x  \; \text{as}\; m \to \infty. \]
}Thus we get
$ \langle I_{\la,\mu}^\prime(u),\phi\rangle= \lim\limits_{m \rightarrow \infty}\langle I_{\la,\mu}^\prime(u_m),\phi\rangle=0$.
{Therefore} $u$ is a weak  solution of $(P_{\la,\mu})$. \\
Let $\widetilde{u_m}= u_m-u$, so by Lemma $2.3$ of \cite{my} we have
\begin{equation}\label{PSlim1}
B(u_m)-B(\widetilde{u_m}) \rightarrow B(u)\; \text{as}\; m \to \infty.
\end{equation}
Also since $I^\prime_{\la,\mu}(u_m) \rightarrow 0$, we get
\begin{equation}\label{PSlim2}
\big((T_\mu-\la)u_m,u_m\big)- B(u_m) \rightarrow 0\; \text{as}\; m \rightarrow \infty.
\end{equation}
Then using \eqref{PSlim1} and \eqref{PSlim2}, we get
\[\lim\limits_{m \rightarrow \infty} \left( \big((T_\mu-\la)\widetilde{u_m}, \widetilde{u_m}\big)- B(\widetilde{u_m})\right)=0.\]
Let $\lim\limits_{m \rightarrow \infty}\big((T_\mu-\la)\widetilde{u_m}, \widetilde{u_m}\big)= \lim\limits_{m \rightarrow \infty} B(\widetilde{u_m})=d\; \text{(say)}$. Then it is easy to show that $I_{\la,\mu}(u)\geq 0$ and using this we get
\[\frac{n+2-\alpha}{2(2n-\alpha)} S_A^{\frac{2n-\alpha}{n+2-\alpha}}>c = \lim\limits_{m \rightarrow \infty} I_{\la,\mu}(u_m) \geq \frac12 \int_{\mb R^n}|\nabla_A\widetilde{u_m}|^2~\mathrm{d}x - \frac{1}{22^*_\alpha}B(\widetilde{u_m})+o_m(1).\]
This implies
\begin{equation}\label{PSlim3}
d \leq \frac{2c(2n-\alpha)}{n+2-\alpha}< S_A^{\frac{2n-\alpha}{n+2-\alpha}}.
\end{equation}
Let $D = \{x \in \mb R^n:\; g(x) \leq M\}$, where $M$ is defined in $(g2)$. Then using similar arguments as in Lemma \ref{comp_lem1}, we can show that
\begin{equation}\label{PSlim4}
\int_D |\widetilde{u_m}|^2~\mathrm{d}x \rightarrow 0 \; \text{as} \; m \rightarrow \infty.
\end{equation}
Then using \eqref{lambdd}, definition of $S_A$ and \eqref{PSlim4} we get
\begin{equation*}
\begin{split}
S_A B(\widetilde{u_m})^{\frac{1}{2^*_\alpha}} & \leq \int_{\mb R^n}|\nabla_A \widetilde{u_m}|^2~\mathrm{d}x\leq \int_{\mb R^n}|\nabla_A \widetilde{u_m}|^2~\mathrm{d}x + \int_{\mb R^n\setminus D} (\mu g(x)-\la)|\widetilde{u_m}|^2~\mathrm{d}x\\
& \leq \big((T_\mu-\la)\widetilde{u_m},\widetilde{u_m}\big) + \la \int_{D}|\widetilde{u_m}|^2~\mathrm{d}x =  \big((T_\mu-\la)\widetilde{u_m},\widetilde{u_m}\big) + o_m(1).
\end{split}
\end{equation*}
Passing on the limits we get
$d\geq S_A^{\frac{2n-\alpha}{n+2-\alpha}}$
which is a contradiction to \eqref{PSlim3}. Therefore, $d=0$ that is $u_m \rightarrow u$ strongly in $E$ as $ m \rightarrow \infty$.\hfill{\QED}
\end{proof}
%%%%%%%%%%%%%%%%%%%%%%%%%%%%%%%%%%%%%%%%%%%%%%%%%%%%%%%%%%%%%%%%%%%%%%%%%%%%%%%%%%%%%%%%%%%%%%%%%%%%%%%%%%%%%%%%%%%%%%%%%%%%%%%%%%%%

\section{Proof of main Theorems}
Before proving the main theorems, we prove some results that will help us to achieve our goal. The theorem below is similar to Theorem $1.1$ of \cite{ArSz}.
\begin{Theorem}\label{S_Aattain}
If $g\geq 0$ and $A \in L^n_{\text{loc}}(\mb R^n, \mb R^n)$, then the infimum $S_A$ is attained if and only if $\text{curl }A \equiv 0$.
\end{Theorem}
\begin{proof}
At first, we prove that $S_A = S_{H,L}$. By Proposition \ref{H-L_ineq} and Theorem \ref{dia_eq}, {for $u \in H^1_A(\mb R^n)$} we have
\begin{equation*}
S_{H,L} \leq \frac{\int_{\mb R^n} |\nabla|u||^2~\mathrm{d}x}{B(u)^{\frac{n-2}{2n-\alpha}}} \leq \frac{\int_{\mb R^n}|\nabla_Au|^2~\mathrm{d}x}{B(u)^{\frac{n-2}{2n-\alpha}}}.
\end{equation*}
This implies $S_{H,L} \leq S_A$. Without loss of generality, we assume $0\in \Om$ and {for some $\delta>0$}, $B(0,\delta) \subset \Om \subset B(0,{2\delta})$($B(0,r)$ is an open ball of radius $r$ and center $0$ ). Let
$$U_\epsilon (x)= (n(n-2))^{\frac{n-2}{4}}\left(\frac{\epsilon}{\epsilon^2+|x|^2}\right)^{\frac{n-2}{4}}$$
and $u_\epsilon(x)= \psi(x)U_\epsilon(x)$, where $\psi \in C_c^{\infty}(\mb R^n,[0,1])$ such that $\psi \equiv 1$ in $B(0,\delta)$ and $\psi \equiv 0$ in $\mb R^n \setminus \Om$. We recall the following asymptotic estimates from section 3 of \cite{my}-
\begin{enumerate}
\item[(1)] $|\nabla u_\epsilon|^2_2 = C(n,\alpha)^{\frac{n(n-2)}{2(2n-\alpha)}}S_{H,L}^{\frac{n}{2}}+ O(\epsilon^{n-2})$.
\item[(2)]  $$ |u_\epsilon|_2^2 \geq
\left\{
	\begin{array}{ll}
	d \epsilon^2 |\ln \epsilon| + O(\epsilon^2) & \mbox{if } n=4 \\
		d\epsilon^2 + O(\epsilon^{n-2}) & \mbox{if } n\geq 5,
	\end{array}
\right.$$
where $d$ is  a positive constant.
\item[(3)] $B(u_\epsilon)^{(n-2)/(2n-\alpha)} \geq \left(C(n,\alpha)^{n/2}S_{H,L}^{(2n-\alpha)/2}- O(\epsilon^{{(2n-\alpha)/{2}}})\right)^{(n-2)/(2n-\alpha)}$.
    \item[(4)] $|u_\epsilon|_{2^*}^{2^*}= S^{n/2}+ O(\epsilon^n)$.
\end{enumerate}
{We have $u_\epsilon$ is bounded in $L^{2^*}(\mb R^n)$ and} $u_\epsilon(x) \rightarrow 0$ a.e. in $\mb R^n$ as $\epsilon \rightarrow 0$.
% So
%$$\int_{\mb R^n}g(x)|u_\epsilon|^2~\mathrm{d}x= \langle g,|u_\epsilon |^2\rangle \rightarrow 0 \;\text{as}\; \epsilon \rightarrow 0,$$
% where the duality product is taken with respect to $L^{n/2}(\mb R^n)$ and $L^{2^*/2}(\mb R^n)$.
Since $A \in L^{n}_{\text{loc}}(\mb R^n, \mb R^n)$  we have
$$\int_{\mb R^n}|A u_\epsilon|^2~\mathrm{d}x = \langle|A|^2, |u_\epsilon|^2 \rangle \rightarrow 0  \; \text{as}\; \epsilon \rightarrow 0 $$
where the duality product is taken with respect to $L^{n/2}(\mb R^n)$ and $L^{2^*/2}(\mb R^n)$.
Now, {let $\delta^\prime>0$ be given then choosing $\epsilon >0$ small enough and using $(1)$ and $(3)$ we get}
\begin{equation*}
\begin{split}
\frac{\int_{\mb R^n}|\nabla_A u_\epsilon|^2~\mathrm{d}x}{B(u_\epsilon)^{\frac{n-2}{2n-\alpha}}} &\leq \frac{\int_{\mb R^n}|\nabla u_\epsilon|^2+ |Au_\epsilon|^2 ~\mathrm{d}x}{B(u_\epsilon)^{\frac{n-2}{2n-\alpha}}}\\
& \leq \frac{C(n,\alpha)^{\frac{n(n-2)}{2(2n-\alpha)}}{S_{H,L}^{n/2}}+ O(\epsilon) }{\left(C(n,\alpha)^{\frac{n}{2}}S_{H,L}^{\frac{2n-\alpha}{2}}- O(\epsilon^{{\frac{2n - \alpha}{2}}})\right)^{\frac{n-2}{2n-\alpha}}}\\
& \leq S_{H,L}+\delta^\prime.
\end{split}
\end{equation*}
This implies $S_A \leq S_{H,L}$, therefore $S_A = S_{H,L}$. Let $u$ be minimizer of $S_A$ normalized by $B(u)=1$. Then \begin{equation}\label{Sattn1}
S_A = \int_{\mb R^n}|\nabla_Au|^2\geq \int_{\mb R^n}|\nabla|u||^2~\mathrm{d}x\geq S_{H,L}.
\end{equation}
Consequently, $|u(x)|= U_\epsilon(x-a)/B(U_\epsilon)$, for some $a \in \mb R^n$ because the minimizers of $S_{H,L}$ are of the form $U_\epsilon$ which are invariant under translation and dilation(Lemma 1.2, 1.3 of \cite{my}).  We can take $|u|>0$ and the equality in \eqref{Sattn1} occurs when the diamagnetic inequality in Theorem \ref{dia_eq} has an equality a.e. Therefore $\text{Im}((\nabla_Au)\overline{u}/|u|)=0$ that is $A = -\text{Im}(\nabla u/u)$. Since $\text{curl}(\nabla u/u)=0$, we are done. The condition is also sufficient, the proof follows similarly as in Theorem $1.1$ of \cite{ArSz}.\hfill{\QED}
\end{proof}

\noi The next step to prove our main theorem is introducing the Nehari manifold. Let
\[\mc N_{\la,\mu}= \left\{u \in E\setminus \{0\}:\; \langle I^\prime_{\la,\mu}(u),u\rangle=0\right\} = \left\{u \in E\setminus \{0\}:\; \big((T_\mu-\la)u,u\big)=B(u) \right\}.\]
Then the critical points of $I_{\la,\mu}$ lie in $\mc N_{\la,\mu}$. Let $X= \{v \in E: \; B(v)=1\}$ then using fibering map analysis, we say that for each $u \in E$, there exist
\[t_u = \left(\frac{((T_\mu-\la)(u),u)}{B(u)}\right)^{\frac{n-2}{2(n+2-\alpha)}}\]
such that $t_uu \in \mc N_{\la,\mu}$. Using Proposition $1.1$ of \cite{Sz}, we get $\mc N_{\la,\mu}$ is radially {diffeomorphic} to $X$ via the map
\[u \mapsto \left(\frac{((T_\mu-\la)(u),u)}{B(u)}\right)^{\frac{n-2}{2(n+2-\alpha)}}u. \]
On $\mc N_{\la,\mu}$,
\[I_{\la,\mu}(u) = \frac{n+2-\alpha}{2(2n-\alpha)}((T_\mu-\la)(u),u),\]
so we get
\[k_{\la,\mu}:= \inf_{u \in \mc N_{\la,\mu}} I_{\la,\mu}(u)= \frac{n+2-\alpha}{2(2n-\alpha)} \inf_{v \in X} ((T_\mu-\la)(v),v)^{\frac{2n-\alpha}{n+2-\alpha}}.\]
%%%%%%%%%%%%%%%%%%%%%%%%%%%%%%%%%%%%%%%%%%%%%%%%%%%%%%%%%%%%%%%%%%%%%%%%%%%%%%%%%%%%%%%%%%%%%%%%%%%%%%%%%%%%%%%%%%%%%%%%%%%%%%%%%%%%%%%%%%%%%%%%%%%%
%\begin{Proposition}
%Let $u \in \mc N_{\la,\mu}$ be a critical point of $I_{\la,\mu}$ such that $I_{\la,\mu}(u) < \frac{2(2^*_\alpha-1)}{2^*_\alpha}k_{\la,\mu}$. Then $u$ does not change sign that is $|u|$ is a solution of $(P_{\la,\mu})$.
%\end{Proposition}
%\begin{proof}
%Since $u$ is a critical point of $I_{\la,\mu}$, for every $w \in E$
%\begin{equation*}
%((T_\mu-\la)(u),w)= \text{Re} \left( \int_{\mb R^n} (|x|^{-\alpha}*|u|^{2^*_\alpha})|u|^{2^*_\alpha}u\overline{w}~\mathrm{d}x\right).
%\end{equation*}
%Let $w= u^{\pm}$ in above equation, where $u^{\pm}= \pm \max\{\pm u,0\}$. Suppose $u$ changes sign, so $u^+$ and $u^-$ are both nonzero and $u^\pm \in \mc N_{\la,\mu}$. Then $I_{\la,\mu}(u) \textcolor{red}{\geq} \frac{2^*_\alpha-1}{2^*_\alpha}( I_{\la,\mu}(u^+)+ I_{\la,\mu}(u^-)) \geq \frac{2(2^*_\alpha-1)}{2^*_\alpha}k_{\la,\mu}$ which is a contradiction. Thus, $u$ does not change sign and we have the conclusion.\hfill{\QED}
%\end{proof}
%%%%%%%%%%%%%%%%%%%%%%%%%%%%%%%%%%%%%%%%%%%%%%%%%%%%%%%%%%%%%%%%%%%%%%%%%%%%%%%%%%%%%%%%%%%%%%%%%%%%%%%%%%%%%%%%%%%%%%%%%%%%%%%%%%%%%%%%%%%%%%%%%%%

Now consider any domain $\mathcal{Q} \subset \mb R^n$. As we defined $I_{\la,\mu}$, in a similar manner we define
\begin{equation*}
\begin{split}
I_{\mu,\mc Q}(u)&= \frac12 \int_{\mc Q}(|\nabla_{A}u|^2+ \la|u|^2)~\mathrm{d}x - \frac{1}{22^*_\alpha}\int_{\mc Q}\int_{\mc Q}\frac{|u(x)|^{2^*_\alpha}|u(y)|^{2^*_\alpha}}{|x-y|^\alpha}~\mathrm{d}x\mathrm{d}y\\
&= \frac12 ((T_0-\la)(u),u)- \frac{1}{22^*_\alpha}B(u)
\end{split}
\end{equation*}
for $u \in H^{0,1}_A(\mc Q)$. This is an energy functional associated to the problem
$$ (P_\la)
\left\{
	\begin{array}{ll}
        (-i\nabla +A(x))^2u = \la u + (|x|^\alpha*|u|^{2^*_\alpha}){|u|^{2^*_\alpha-2}u}, \; u>0 & \mbox{in }  \mc Q \\
        u=0 & \mbox{on } \partial \mc Q.
	\end{array}
\right.$$
The Nehari manifold associated to $(P_\la)$ with $\mc Q= \Om$ is given by
\[\mc N_{\la,\mc Q}= \left\{u \in H^{0,1}_A(\mc Q)\setminus \{0\}:\; ((T_0-\la)(u),u)= B(u) \right\}\]
which is radially {diffeomorphic} to $X_{\mc Q}= \{v \in H^{0,1}_A(\mc Q):\; B(v)=1\}$. We set
\[k_{\mu,\mc Q}:= \inf_{u \in \mc N_{\la,\mc Q}} I_{\la,\mc Q}(u)= \frac{n+2-\alpha}{2(2n-\alpha)}\inf_{v \in X_{\mc Q}} ((T_0-\la)(u),u)^{\frac{2n-\alpha}{n+2-\alpha}}.\]
%%%%%%%%%%%%%%%%%%%%%%%%%%%%%%%%%%%%%%%%%%%%%%%%%%%%%%%%%%%%%%%%%%%%%%%%%%%%%%%%%%%%%%%%%%%%%%%%%%%%%%%%%%%%%%%%%%%%%%%%%%%%%%%%%%%%%%%%%%%%%%%%%%
\begin{Lemma}\label{inf_lev}
If $\la \in (0,\la_1(\Om))$ and $\mu \geq \mu(\la)$, then the following holds
\[\frac{n+2-\alpha}{2(2n-\alpha)}(\beta_\la S_A)^{\frac{2n-\alpha}{n+2-\alpha}} \leq k_{\la,\mu} \leq k_{\la,\Om} < \frac{n+2-\alpha}{2(2n-\alpha)} S_A^{\frac{2n-\alpha}{n+2-\alpha}}\]
{where $\beta_{\la}$ is defined as in Lemma \ref{comp_lem2}.}
\end{Lemma}
\begin{proof}
Using Lemma \ref{comp_lem2}, we have $\beta_\la \|v\|_A^2 \leq \beta_\la \|v\|_\mu^2 \leq ((T_\mu-\la)(u),u)$. This implies, taking infimum over $X$, we get
\[\frac{n+2-\alpha}{2(2n-\alpha)}(\beta_\la S_A)^{\frac{2n-\alpha}{n+2-\alpha}} \leq k_{\la,\mu}.\]
This gives the first inequality. Now, for the second inequality, since \[X_\Om =\left\{u \in H_A^{0,1}(\Om):\;  \int_{\Om}\int_{\Om}\frac{|u(x)|^{2^*_\alpha}|u(y)|^{2^*_\alpha}}{|x-y|^{\alpha}}~\mathrm{d}x\mathrm{d}y=1\right\}\subset X,\] where $\Om=$ interior of $g^{-1}(0)$, we get $k_{\la,\mu} \leq k_{\la,\Om}$. We aim to show that
\[k_{\la,\Om} <  \frac{n+2-\alpha}{2(2n-\alpha)} S_A^{\frac{2n-\alpha}{n+2-\alpha}}.\]
Let $ U_\epsilon$ and $u_\epsilon$ be as defined in Lemma \ref{S_Aattain}. Define
\[J_\la (u)= \frac{\int_{\Om}(|\nabla_Au|^2-\la |u|^2)~\mathrm{d}x}{\left( \int_{\Om}\int_\Om\frac{|u(x)|^{2^*_\alpha}|u(y)|^{2^*_\alpha}}{|x-y|^\mu}~\mathrm{d}x\mathrm{d}y\right)^{\frac{n-2}{2n-\alpha}}}.\]
Let $A$ be continuous at $0$ and $\gamma(x):= -\sum A_j(0)x_j$. Then it is easy to check that $(A+\nabla \gamma)(0)=0$ and therefore by continuity of $A$ at $0$ we get a $\delta_1 >0$ such that
\[|(A+\nabla \gamma)(x)|^2\leq \tilde{k} <\la,\; \text{for all}\; |x|<\delta_1.\]
Also {let $\delta_2= \min\{\delta,\delta_1\}$} and define $v_\epsilon(x)= \psi(x)U_\epsilon(x)\exp(i\gamma(x)) $, where
$$\psi(x)=
\left\{
	\begin{array}{ll}
1 & \mbox{in }  B(0,{\delta_2/2}) \\
0  & \mbox{in } \mb R^n \setminus \Om.
	\end{array}
\right.$$
Then using $(1)$ of Lemma \ref{S_Aattain}, we get
\begin{equation*}
\begin{split}
\int_{\mb R^n}(|\nabla_A v_\epsilon|^2- \la |v_\epsilon|^2)~\mathrm{d}x &= \int_{\mb R^n}(|(-i\nabla+A)(\psi U_\epsilon \exp(i\gamma))|^2-\la \psi^2U_\epsilon^2 )~\mathrm{d}x\\
& \leq \int_{\mb R^n}(|\nabla(\psi U_\epsilon)|^2+\psi^2U_\epsilon^2|\nabla \gamma+A|^2-\la\psi^2U_\epsilon^2 )~\mathrm{d}x\\
& \leq C(n,\alpha)^{\frac{n(n-2)}{2(2n-\alpha)}}S_{H,L}^{\frac{n}{2}}+ O(\epsilon^{n-2})+({\tilde{k}-\la})\int_{B\left(0,\frac{\delta_2}{2}\right)}U_\epsilon^2~\mathrm{d}x.
\end{split}
\end{equation*}
Moreover, using $(3)$ of Lemma \ref{S_Aattain}, we get
\begin{equation*}
B(v_\epsilon )= B(u_\epsilon ) \geq C(n,\alpha)^{n/2}S_{H,L}^{(2n-\alpha)/2}- o(\epsilon^{{(2n - \alpha)/2}}).
\end{equation*}
It is a standard result that for $\epsilon >0$ small enough
$$\int_{B\left(0,\frac{\delta_2}{2}\right)}U_\epsilon^2 ~\mathrm{d}x \geq
\left\{
	\begin{array}{ll}
C\epsilon^2|\log \epsilon| & \mbox{if }  n=4 \\
C\epsilon^2 & \mbox{if } n\geq 5,
	\end{array}
\right.$$
where $C>0$ is a constant depending only on $n$.  Therefore since $\tilde{k}-\la <0$, when $n \geq 5$, for $\epsilon>0$ small enough we get
\begin{equation}\label{ngeq4}
J_\la(v_\epsilon) \leq \frac{C(n,\alpha)^{\frac{n(n-2)}{2(2n-\alpha)}}S_{H,L}^{\frac{n}{2}}+ O(\epsilon^{n-2})+({\tilde{k}-\la})C\epsilon^2 }{\left(C(n,\alpha)^{n/2}S_{H,L}^{(2n-\alpha)/2}- O(\epsilon^{{(2n - \alpha)/2}})\right)^{(n-2)/(2n-\alpha)}} < S_{H,L} = S_A.
\end{equation}
This implies
\[k_{\la,\Om}\leq \frac{n+2-\alpha}{2(2n-\alpha)} J_\la(v_\epsilon)^{\frac{2n-\alpha}{n+2-\alpha}} < \frac{n+2-\alpha}{2(2n-\alpha)} S_A^{\frac{2n-\alpha}{n+2-\alpha}}\]
that is the last inequality holds for $n\geq 5$. Similarly the result follows for $n=4$.\hfill{\QED}
\end{proof}

{\begin{Remark}
For the case $n=3$, \eqref{ngeq4} becomes
% it is easy to get that
%\[\int_{B\left(0,\frac{\delta_2}{2}\right)}U_\e^2 \geq C\e  \]
%for some constant $C>0$. So in this case,
\begin{equation}\label{ngeq41}
J_{\la}(v_\e) \leq S_{H,L} - \frac{((\la-\tilde k)-O(1))C\e}{\left(C(\alpha)^{3/2}S_{H,L}^{(6-\alpha)/2}- O(\epsilon^{{(6 - \alpha)/2}})\right)^{1/(6-\alpha)}}
\end{equation}
where the right hand side of \eqref{ngeq41} becomes less than $S_{H,L}$ if $\la>0$ is chosen large enough. But since $\la \in (0,\la_1)$ this can not be possible. So we remark that the question of existence and concentration of solutions of $(P_{\la,\mu})$ remains open in the case $n=3$.
\end{Remark}}

\noi \textbf{Proof of Theorem \ref{MT1}:} Let $\{u_m\}$ be a minimizing sequence for $I_{\la,\mu}$ on $\mc N_{\la,\mu}$. Then by Ekeland Variational Principle \cite{Eke}, $\{u_m\}$ becomes a Palais-Smale sequence. Using Proposition \ref{comp_lem4} and Lemma \ref{inf_lev}, we conclude that there exist a subsequence of $\{u_m\}$ that converges to least energy solution, say $u_\mu$ of $(P_{\la,\mu})$. \hfill{\QED}

\noi \textbf{Proof of Theorem \ref{MT2}:} Let $\{u_m\}$ be a sequence of solution for the problem $(P_{\la,\mu_m})$ such that $\la \in (0, \la_1(\Om))$, $\mu_m \rightarrow \infty$ and
\[I_{\la,\mu_m}(u_m)= \big((T_{\mu_m}-\la)(u_m),u_m\big) \rightarrow c < \frac{n+2-\alpha}{2(2n-\alpha)} S_{H,L}^{\frac{2n-\alpha}{n+2-\alpha}}.\]
By Lemma \ref{comp_lem2}, $((T_{\mu_m}-\la)u_m,u_m) \geq \beta_\la \|u_m\|_{\mu_m}^2$ {for sufficiently large $m$} which implies that $\beta_\la \|u_m\|_{\mu_m}^2 \leq c+ o(1)$. So, $\{u_m\}$ is bounded in $E$ and using Lemma \ref{comp_lem1}, we say that there exists $u\in H^{0,1}_A(\Om)$ such that, upto a subsequence, $u_m \rightharpoonup u $ weakly in $E$ and $u_m \rightarrow u$ strongly in $L^2(\mb R^n)$. Since $u_{\mu_m}$ solves $(P_{\la,\mu_m})$, we have
\begin{equation}\label{mainthr3}
\text{Re}\left( \int_{\mb R^n} \nabla_A u_m \overline{\nabla_A v}~\mathrm{d}x + \int_{\mb R^n}\left(\mu_m g(x)-\la )u_m\overline{v}- (|x|^{-\alpha}* |u_m|^{2^*_\alpha})|u_m|^{2^*_\alpha-2}u_m \overline{v}\right)~\mathrm{d}x \right)=0
\end{equation}
for every $v \in E$. Since $\Om = \{x \in \mb R^n:\; g(x)=0\}$, for any $v \in H^{0,1}_A(\Om)$, $\mu_m\int_{\mb R^n} g(x)u_m\overline{v}~\mathrm{d}x=0$. Letting $m \rightarrow \infty$ in \eqref{mainthr3}, we obtain
\[\text{Re}\left( \int_{\mb R^n} \nabla_A u \overline{\nabla_A v}~\mathrm{d}x -\la \int_{\mb R^n} u\overline{v}~\mathrm{d}x- \int_{\mb R^n} (|x|^{-\alpha}* |u|^{2^*_\alpha})|u|^{2^*_\alpha-2}u \overline{v}~\mathrm{d}x \right)=0, \]
{ for all $v \in H^{0,1}_A(\Om)$} which implies that $u$ is a weak solution of $(P_\la)$. Let $\tilde{u}_m= u_m -u$, then $\tilde{u}_m \rightharpoonup 0$ weakly in $E$ and $\tilde{u}_m \rightarrow 0$ strongly in $L^2(\mb R^n)$. Therefore,
\begin{equation}\label{mainthr4}
\big((T_{\mu_m}-\la){u}_m, {u}_m\big) = \big((T_{\mu_m}-\la)\tilde{u}_m, \tilde{u}_m\big) + \big((T_{0}-\la)u,u\big) + o(1).
\end{equation}
By Lemma $2.3$ of \cite{my}, we get
\[ B(u_m) -B(\tilde{u}_m) \rightarrow B(u) \; \text{as} \; m \rightarrow \infty.  \]
{Since $u$ is a weak solution of $(P_\la)$} it is easy to see that
\begin{equation}\label{mainthr1}
\big((T_{\mu_m}-\la)\tilde{u}_m, \tilde{u}_m\big) - B(\tilde{u}_m) = o(1).
\end{equation}
We claim that $B(\tilde{u}_m) \rightarrow 0$ as $m \rightarrow \infty$. Suppose not, that is $B(\tilde{u}_m) \rightarrow d >0$ as $m \rightarrow \infty$. Using \eqref{mainthr1} and arguments as in Lemma \ref{comp_lem4}, we get
\begin{equation*}
S_A B(\tilde{u}_m)^{\frac{n-2}{2n-\alpha}} \leq \int_{\mb R^n} |\nabla_A\tilde{u}_m|^2~\mathrm{d}x \leq \big((T_{\mu_m}-\la)\tilde{u}_m,\tilde{u}_m\big)+o(1)= B(\tilde{u}_m)+o(1).
\end{equation*}
Consequently, we get that $S_A \leq B(\tilde{u}_m)^{\frac{n+2-\alpha}{2n-\alpha}}+o(1)$ which implies
\begin{equation}\label{mainthr2}
S_A^{\frac{2n-\alpha}{n+2-\alpha}} \leq \lim_{m\rightarrow \infty}B(u_m).
\end{equation}
It is not hard to find that $\lim\limits_{m \rightarrow \infty} B(u_m)= \frac{2c(2n-\alpha)}{(n+2-\alpha)}$. Thus from \eqref{mainthr2} we get
\[S_A^{\frac{2n-\alpha}{n+2-\alpha}} \leq \lim\limits_{m\rightarrow \infty} B(u_m) < S_A^{\frac{2n-\alpha}{n+2-\alpha}}\]
which is a contradiction. Hence, $\lim\limits_{m\rightarrow \infty} B(\tilde{u}_m)= 0$ and $\big((T_{\mu_m}-\la)\tilde{u}_m, \tilde{u}_m\big) \rightarrow 0$ as $m \rightarrow \infty$. Using \eqref{mainthr4}, we get
\begin{equation}\label{mainthr5}
\big((T_{0}-\la)u, u\big)= \lim\limits_{m \rightarrow  \infty}\big((T_{\mu_m}-\la){u}_m,{u}_m\big).
 \end{equation}
{Since} $u \equiv 0$ in $\mb R^n \setminus \Om$, so $u_m = \tilde{u}_m$ in $\mb R^n \setminus \Om$. Also since $g\equiv 0$ in $\Om$, {for sufficiently large $m$ we get}
\[\int_{\mb R^n}g(x)u_m^2~\mathrm{d}x \leq \mu_m \int_{{\mb R^n \setminus\Om}}g(x)\tilde{u}_m^2~\mathrm{d}x \leq ((T_{\mu_m}-\la)(\tilde{u}_m), \tilde{u}_m)  + o(1). \]
Therefore, $\int_{\mb R^n}g(x) u_m^2~\mathrm{d}x=0$ and {since $u_m \to u$ strongly in $L^2(\mb R^n)$}, \eqref{mainthr5} implies
\[\lim_{m \rightarrow \infty}\int_{\mb R^n}|\nabla_A u_m|^2~\mathrm{d}x = \int_{\mb R^n}|\nabla_A u|^2~\mathrm{d}x \]
that is $u_m \rightarrow u$ {strongly in $E$ as $m \to \infty$}.\hfill{\QED}

\begin{Corollary}
If $\la \in (0,\la_1(\Om))$, then $\lim\limits_{\mu \rightarrow \infty} k_{\la,\mu} = k_{\la,\Om}$.
\end{Corollary}
\begin{proof}
By Lemma \ref{inf_lev}, $k_{\la,\mu} \rightarrow a \leq k_{\la,\Om}< \frac{n+2-\alpha}{2(2n-\alpha)}S_A^{\frac{2n-\alpha}{n+2-\alpha}} $ as $\mu \to \infty$. Theorem \ref{MT1} implies that $k_{\la,\mu}$ is achieved for $\mu \geq \mu(\la)$. Therefore, Theorem \ref{MT2} says $a$ must be achieved by $I_{\mu,\Om}$ on $\mc N_{\la,\Om}$. Hence $a \geq k_{\mu,\Om}$.\hfill{\QED}
\end{proof}

\section{{A Remark on nonlocal counterpart of $(P_{\la,\mu})$}}
In this section, we brief the nonlocal extension of the problem $(P_{\la,\mu})$ given by
\begin{equation*}
(P_{\la,\mu}^s)\left\{
\begin{split}
& (-\De)^s_A u + \mu g(x)u = \la u + (|x|^{-\alpha} * |u|^{2^*_{\alpha,s}})|u|^{2^*_{\alpha,s}-2}u  \;\text{in} \; \mb R^n ,\\
 & u \in H^s_A(\mb R^n, \mb C)
\end{split}
\right.
\end{equation*}
where $n \geq 4s$, $s \in (0,1)$ and $\alpha\in (0,n)$. Here $2^*_{\alpha,s}=\textstyle \frac{2n-\alpha}{n-2s}$ is the critical exponent in the sense of Hardy-Littlewood-Sobolev inequality. We assume the same conditions on $A$ and $g$ as before. For $u \in C_c^\infty(\Om)$, the fractional magnetic operator $(-\De)^s_A$, up to a normalization constant, is defined by
\[(-\De)^s_A u (x) = 2\lim_{\e \to 0^+} \int_{\mb R^n \setminus B_{\e}(x)} \frac{u(x)-e^{i(x-y)\cdot A\left(\frac{x+y}{2}\right)}u(y) }{|x-y|^{n+2s}}\mathrm{d}y \]
for all $x \in \mb R^n$. Up to correcting the operator by the factor $(1-s)$, it is true that $(-\De)^s_A$ converges to $-\De_A$ as $s \uparrow 1$. For further details we refer \cite{PM} and references therein. \\

\noi \textbf{Functional Setting-} Let $L^2_g(\mb R^n,\mb C)$ denote the Lebesgue space of complex valued functions with $\textstyle \int_{\mb R^n}g|u|^2 < +\infty$ endowed with the real scalar product
\[\langle u,v \rangle_{L^2_g}:= \text{Re}\left( \int_{\mb R^n} g(x)u\overline{v}~\mathrm{d}x \right), \; \text{for}\; u,v \in L^2_g(\mb R^n, \mb C).\]
We consider the magnetic Gagliardo semi-norm defined by
\[[u]_{s,A}^2 := \int_{\mb R^n}\int_{\mb R^n} \frac{\left|u(x)-e^{i(x-y)\cdot A\left(\frac{x+y}{2}\right)}u(y)\right|^2}{|x-y|^{n+2s}} \mathrm{d}x\mathrm{d}y.\]
 The scalar product is defined as
\begin{align*}
&\quad \langle u,v \rangle_{s,A} \\
&:= \langle u,v \rangle_{L^2_g}+ \text{Re}\left( \int_{\mb R^n}\int_{\mb R^n} \frac{\left(u(x)-e^{i(x-y)\cdot A\left(\frac{x+y}{2}\right)}u(y)\right)\overline{\left(v(x)-e^{i(x-y)\cdot A\left(\frac{x+y}{2}\right)}v(y)\right)}}{|x-y|^{n+2s}} \mathrm{d}x\mathrm{d}y\right)
\end{align*}
and the corresponding norm is given by
\[\|u\|_{s,A}= \left(\|u\|_{L^2_g}^2 + [u]_{s,A}^2 \right)^{\frac12}.\]
We define $H^s_{A,g}(\mb R^n,\mb C)$ as the closure of $C_c^\infty(\mb R^n, \mb C)$ with respect to the norm $\|\cdot\|_{s,A}$. Then we have the following properties regarding the function space $H^s_{A,g}(\mb R^n,\mb C)$ -
\begin{enumerate}
\item[(i)] $(H^s_{A,g}(\mb R^n,\mb C), \langle \cdot,\cdot\rangle_{s,A})$ is a real Hilbert space.
\item[(ii)] The embedding
$H^s_{A,g}(\mb R^n,\mb C) \hookrightarrow L^p(\mb R^n,\mb C)$
is continuous for all $p \in [2,2^*_s]$ where $2^*_s = \textstyle \frac{2n}{n-2s}$. Furthermore, for any $K \Subset \mb R^n$ and $p \in [1,2^*_s)$, the embedding $ H^s_{A,g}(\mb R^n,\mb C) \hookrightarrow L^p(K,\mb C)$ is compact.
\item[(iii)] (\textit{Diamagnetic inequality}) For each $u \in H^s_{A,g}(\mb R^n,\mb C)$
\[|u|\in H^s_g(\mb R^n, \mb C) \; \text{and} \;\||u|\|_s\leq \|u\|_{s,A}\]
where $H^s_g(\mb R^n,\mb C) = H^s_{A,g}(\mb R^n,\mb C)$ with $A \equiv 0$.
\end{enumerate}
For further details related to this topic, we refer \cite{PM,MPS} and the references therein.
\begin{Definition}
We say that $u \in H^s_{A,g}(\mb R^n, \mb C)$ is a weak solution of $(P_{\la,\mu}^s)$ if
\begin{align*}
& \text{Re}\left( \int_{\mb R^n}\int_{\mb R^n} \frac{\left(u(x)-e^{i(x-y)\cdot A\left(\frac{x+y}{2}\right)}u(y)\right)\overline{\left(v(x)-e^{i(x-y)\cdot A\left(\frac{x+y}{2}\right)}v(y)\right)}}{|x-y|^{n+2s}} \mathrm{d}x\mathrm{d}y + \mu \int_{\mb R^n} g(x)u\overline v~\mathrm{d}x\right.\\
&\quad \quad \left.- \la \int_{\mb R^n}u\overline{v}~\mathrm{d}x - \int_{\mb R^n}(|x|^{-\mu}\ast |u|^{2^*_{\alpha,s}})|u|^{2^*_{\alpha,s}-2}u\overline{v}~\mathrm{d}x\right) =0
\end{align*}
for all $v \in H^s_{A,g}(\mb R^n, \mb C)$.
\end{Definition}
The functional $I_s: H^s_{A,g}(\mb R^n, \mb C) \to \mb R $, associated to $(P_{\la,\mu}^s)$, is defined by
\[ I_s(u) = \frac{\|u\|_{s,A}^2}{2} -\frac{\la}{2}\int_{\mb R^n}|u|^2~\mathrm{d}x  - \frac{1}{22^*_{\alpha,s}}\int_{\mb R^n}(|x|^{-\mu}\ast |u|^{2^*_{\alpha,s}})|u|^{2^*_{\alpha,s}}~\mathrm{d}x.\]
Then $I_s \in C^1(H^s_{A,g}(\mb R^n, \mb C), \mb R)$ and the critical points of $I_s$ are exactly the weak solutions of $(P_{\la,\mu}^s)$. Based on this setting, we expect that Theorem \ref{MT1} and \ref{MT2} type of results {can as well as} obtained for the problem $(P_{\la,\mu}^s)$ employing the same {arguments as in this article}. Lastly, we cite \cite{MRZ,FPV} for readers as very recent articles concerning this topic.

\section*{Acknowledgments} We thank the anonymous referee {for advising} us to mention the possibility of investigating our results
in the framework of nonlocal magnetic operator which we have addressed in section $5$ and bringing into our notice references \cite{PM,MPS}.

 \linespread{0.5}

\end{document}